\documentclass[11pt,letterpaper]{amsart}
\usepackage{enumitem}
\usepackage{breqn}
\usepackage{comment}
\usepackage{setspace}
\usepackage{bm}
\usepackage{amssymb,amsmath,xcolor,amsthm}
\usepackage{mathtools}
\usepackage{cleveref}
\mathtoolsset{showonlyrefs}
\usepackage{graphicx}

\newcommand{\bi}{\bar{i}}
\newcommand{\bj}{\bar{j}}
\newcommand{\bk}{\bar{k}}
\newcommand{\bl}{\bar{l}}

\newcommand{\bp}{\bar{p}}

\newcommand{\sbpartial}{\bar{\partial}^*}
\newcommand{\spartial}{\partial^*}

\newcommand{\bpartial}{\bar{\partial}}
\newcommand{\bnabla}{\overline{\nabla}}
\newcommand{\pbp}{\partial \bar{\partial}}

\newcommand{\Ker}{\mbox{Ker}}

\newcommand{\ol}{\overline}

\newtheorem{theorem}{Theorem}[section]
\newtheorem{claim}{Claim}[section]
\newtheorem{lemma}[theorem]{Lemma}

\newtheorem{corollary}[theorem]{Corollary}

\newtheorem{definition}[theorem]{Definition}

\numberwithin{equation}{section}

\begin{document}

\title[Monge-Amp\`ere-type equation and Demailly's Morse inequality]{Monge-Amp\`ere-type equation for forms of positive degree and Demailly's transcendental Morse inequality}

\author{Mathew George}
\date{}

\begin{abstract}
    The Monge-Ampère-type equation for forms of positive degree was introduced by Dinew and Popovici to prove the qualitative part of an analytic version of Demailly's transcendental Morse inequality for higher cohomology classes, conditional on the solvability of this nonlinear PDE. In this paper, we show that this can be proven unconditionally. We first show that the originally proposed Laplacian trace condition, $\Lambda^{m-2}\Delta u = 0$, is analytically too rigid to permit solutions and hence introduce a general gauge-fixing that leads to the $(a,b)$ Monge-Ampère-type equation. By deriving \emph{a priori} estimates, we establish the solvability of this class of equations for different parametric values $(a,b)$. As a geometric application, for $b=0$ this framework reduces to the classical complex Monge-Amp\`ere equation, which yields the qualitative Demailly's inequality for higher-degree forms. Further, we prove that the uniqueness of these solutions can be significantly strengthened under the kernel condition $\bar{\partial}^* u = 0$.
\end{abstract}

\subjclass[2020]{Primary 32W20; Secondary 32U40, 35J60}

\maketitle

\bigskip

\section{Introduction}

A central theme in complex geometry is the detection of strictly positive currents using purely cohomological volume thresholds. Demailly’s transcendental Morse inequalities \cite{Dem2010} provide such a criterion bypassing explicit constructions of positive currents. Dinew and Popovici \cite{DP25} introduced Monge-Ampère-type equation for forms of positive degree to extend these bigness criteria to higher-degree $(m,m)$-forms. Specifically, on an $n$-dimensional compact Kähler manifold $(M, \omega)$, for $ 1< m < n$ the Monge-Amp\`ere-type equation for forms of positive degree can be written as

\begin{equation}\label{pde}
    \left[\star_{\omega}\left((\alpha + i \pbp u)\wedge \omega_{n-m-1}\right)\right]^n = dV,
\end{equation}

\

\noindent for an $(m,m)$ form $\alpha >0$ and $dV$ a smooth volume form on $M$. The classical version \cite{BDPP13,Popovici16} of Demailly's transcendental Morse inequality for $(1,1)$ forms states that if $\alpha$ and $\beta$ are nef classes on a compact $n$-dimensional K\"ahler manifold, the difference $[\alpha] - [\beta]$ contains a strictly positive K\"ahler current, given that the following volume threshold holds

$$\int_M \alpha^n - n \int_M \alpha^{n-1} \wedge \beta >0.$$

\

Originally stated in \cite[Conjecture $10.1$, $(ii)$]{BDPP13}, weaker versions were shown in \cite{BDPP13} and \cite{Xiao2015}. The qualitative part of the conjecture was resolved by Popovici \cite{Popovici16}. This inequality and its general version are crucial ingredients for extending the fundamental duality between the pseudo-effective cone $\mathcal{E}$ and the cone of movable classes $\mathcal{M}$ on projective manifolds to compact K\"ahler manifolds. In \cite{DP25}, this is generalized to higher-degree $(m,m)$-forms assuming the solvability of the Monge-Amp\`ere-type equation for forms of positive degree under some gauge conditions. Specifically, it was shown that for a strictly strongly positive closed $(m,m)$-form $\beta$, the class $[(\alpha_{\omega})_m - \beta]_{BC}$ contains a strictly positive current if the integral inequality

\

$$ \frac{1}{n!}\int_M \alpha_{\omega}^n - \frac{1}{(n-m)!}\int_M\alpha_{\omega}^{n-m}\wedge \beta >0  $$

\

\noindent holds, assuming that $\alpha$ is a $d$-closed $(m,m)$ form satisfying $\Delta \alpha_{\omega} = 0$. To construct this current, the solvability of the Monge-Ampère-type equation \eqref{pde} is assumed under

$$    \alpha + i \pbp u >0, \quad u \in \Ker (\spartial) \cap \Ker( \bar{\partial}^*), \quad \text{ and }\Lambda^{m-2} \Delta u =0.$$

\

This set of  conditions is over-constraining and cannot produce solutions except in the trivial case when $dV$ is a constant multiple of $(\alpha_{\omega})^n$. In this paper, we replace $\Lambda^{m-2} \Delta u =0$ by more general condition of the form

\begin{equation}\label{gen-IC}
    (m-1)\Lambda^{m-2} \Delta u = (1+a) i \pbp \Lambda^{m-1}  u +  (1+b) \Lambda^{m-1} \Delta u \;\omega.
\end{equation} 

\

\noindent for any two real constants $a$ and $b$ satisfying $a + nb = m-n-2$. This is a natural gauge-fixing for equation \eqref{pde} as it constrains the relationship between the three $(1,1)$ forms that govern the equation: the complex Hessian $i \pbp \Lambda^{m-1}u$, the form $(\Lambda^{m-1} \Delta u) \; \omega$ and the $(1,1)$ part of the Laplacian $\Delta u$.

\

 Denote the primitive coordinates of the solution $u \in C^{\infty}_{m-1,m-1}(M, \mathbb R)$ of \eqref{pde} by $\psi_k$, for $0 \leq k \leq m-1$. Then it can be shown that equation \eqref{pde} depends only on the first two primitives $\psi_0$ and $\psi_1$. The condition $\Lambda^{m-2} \Delta u = 0$ would imply that $\psi_0$ and $\psi_1$ are harmonic. So $\psi_0$ is constant and $\Delta \psi_1 = 0$, which is incompatible with \eqref{pde}.

 \

Similar to \cite{DP25}, we begin our study of the PDE under condition $(1)$: the solution $u \in \Ker (\spartial) \cap \Ker( \bar{\partial}^*)$. Under $(1)$, it was shown that if the primitives of all orders greater than $0$ are equal, any two solutions are unique up to a factor of $c \omega^n$. In this paper, we show that the condition $(1)$ yields a stronger uniqueness theorem which is stated in Section $2$. This result relies on a cascade of relations between successive primitives  \eqref{cascade} of the solution $u$ coming from the condition $\partial^* u = 0$.

Independent of the geometric constraints required by the transcendental Morse inequality, we establish \emph{a priori} second-order and gradient estimates for this generalized class of Monge-Amp\`ere-type equations. The following theorem can be stated as a consequence. 

\begin{theorem}\label{exist-theorem}
  For any $(a,b)$ with $a+nb = m-n-2$ satisfying condition \ref{a} from Theorem \ref{ab-MAeqn-theorem}, there exists a form $u \in C^{\infty}_{m-1,m-1}(M, \mathbb R)$ and  a unique constant $c > 0$ satisfying

$$\left[\star_{\omega}\left((\alpha + i \pbp u)\wedge \omega_{n-m-1}\right)\right]^n = c \cdot dV,$$

\noindent with  

$$   \alpha + i \pbp u >0, \quad u \in \Ker (\spartial) \cap \Ker( \bar{\partial}^*),$$

\noindent and

$$ (m-1)\Lambda^{m-2} \Delta u = (1+a) i \pbp \Lambda^{m-1}  u +  (1+b) \Lambda^{m-1} \Delta u \;\omega.$$

\

If $(a,b)$ and $\alpha$ satisfies condition \ref{b} in Theorem \ref{ab-MAeqn-theorem}, with $X = \star_{\omega}(\alpha \wedge \omega_{n-m-1})$, then there exists a constant $\epsilon>0$ such that if $|dV - \log\det( \star_{\omega}(\alpha \wedge \omega_{n-m-1})) \omega^n|_{C^{0,\gamma}}< \epsilon$ for some $\gamma >0 $, then for some $c>0$ there exists a solution $u$ for this system.

\end{theorem}

\

The solution $u$ cannot be guaranteed to be unique, although its lowest order primitive coordinates $\psi_0$ and $\psi_1$ are unique up to harmonic forms. Further uniqueness of higher order primitives could only be established under additional assumptions as in Theorem \ref{unique-thm}. As a consequence of the existence of such a form $u \in C^{\infty}_{m-1,m-1}(M, \mathbb R)$ for $b=0$ and $a = m-n-2 $, the Demailly's transcendental Morse inequality for higher cohomologies (Theorem \ref{demailly-thm}) can now be proven.

Independent of the geometric application, we study this equation for other parametric values $(a,b)$. In Section \ref{section3}, this system is reduced to the {\em $(a,b)$ Monge-Amp\`ere-type equation}: 

$$    \left(\star_{\omega}(\alpha \wedge \omega_{n-m-1}) + \frac{n!}{(n-m+1)!} \left[ a i\pbp f + b \Delta f \omega\right] \right)^n = dV.$$

\

We prove the following.

\begin{theorem} \label{ab-MAeqn-theorem}
For any pair $(a,b) \neq (0,0)$, consider the equation

\begin{equation}\label{ab-MAeqn1}
    \log\det(\tilde g_f) = c.\psi(z), \text{ with }\;\;\; \tilde g_f >0.
\end{equation}

\noindent where

$$\tilde g_f = X + \frac{n!}{(n-m+1)!} \left[ a i\pbp f + b \Delta f \omega\right],$$

\

\noindent with $X$ a positive $(1,1)$ form and $\psi(z) \in C^{\infty}(M)$ a positive function. Then the following statements are true.

\begin{enumerate}[label=\Roman*.]
    \item \label{a} If $a+b\geq 0$ and $b \geq 0$, \textbf{or} if $a+b \leq 0$ and $b \leq 0$, then for some $c>0$, there exists a solution $f$ for \eqref{ab-MAeqn1} unique up to constants.

    \item \label{b} If $a+b < 0$, $b > 0$ \textbf{or} $a+b >0$, $b < 0$, and $\kappa(X) < \dfrac{|a|}{n|b|}$, then there exists a constant $\epsilon>0$ such that if $|\psi - \log\det(X)|_{C^{0, \gamma}}< \epsilon$ for some $\gamma >0 $, then for some $c>0$ there exists a solution $f$ for \eqref{ab-MAeqn1} unique up to constants. Here $\kappa(X)$ denotes the condition number of $X$.
\end{enumerate}

\end{theorem}

\

In Section $2$, we establish the strengthened uniqueness theorem under the Kernel condition and also prove that the original Laplace Trace condition is too rigid. In Section $3$, we introduce the $(a,b)$ Monge-Amp\`ere-type equation, and show the application to analytical Demailly's Morse inequality. Finally in Section $4$, \emph{a priori} estimates are shown which implies Theorem \ref{exist-theorem} and Theorem \ref{ab-MAeqn-theorem}.

\section*{Acknowledgement and AI tool use}

The author would like to thank Sławomir Dinew for helpful communications. 

Gemini was used for literature search and improved writing at a few places.

\section{The Kernel Condition and Uniqueness}

\bigskip

The following notations will be used in this paper. For any $(m-1,m-1)$ form $\alpha$ and any $(1,1)$ form $\beta$

\begin{itemize}

    \item $\alpha_{\omega} = \star_{\omega}(\alpha \wedge \omega_{n-m-1})$ defines a $(1,1)$ form.

    \

    \item $\beta_{k}=\dfrac{\beta^k}{k!}.$

    \

    \item The $\partial$-Laplacian of the K\"ahler metric $\omega$ will simply be written $\Delta$ which is defined by $\partial \spartial + \spartial \partial$.

    \

    \item $\Lambda$ and $L$ denote the Trace operator and the Lefschetz operator for the metric $\omega$ respectively.
\end{itemize}

We recall the definition of strong positivity.

\begin{definition}
    An $(m,m)$ form $\alpha \in C^{\infty}_{m,m}(M,\mathbb R)$ is called {\bf strongly positive} if it can be decomposed as 

     $$\alpha =  \sum_{s} c_s \, (i\alpha^s_1 \wedge \bar{\alpha^s}_1) \wedge \dots \wedge (i\alpha^s_m \wedge \bar{\alpha^s}_m)$$

     \noindent for $c_s >0$ and $\alpha^s_i$ a $(1,0)$ form. We denote this by $\alpha >0$. If $c_s \geq 0$, then $\alpha$ is strongly semi-positive, denoted by $\alpha \geq 0$. 
\end{definition}

Given an $\alpha \in C^{\infty}_{m,m}(M, \mathbb R)$ with $\alpha >0$, the Monge-Amp\`ere equation for forms of positive degree for $u \in C_{m-1,m-1}^{\infty}(M, \mathbb R)$ can be written as

\begin{equation}\label{MA-eqn}
    \left[\star_{\omega}\left((\alpha + i \pbp u)\wedge \omega_{n-m-1}\right)\right]^n = dV
\end{equation}

\ 

\noindent assuming $\alpha + i \pbp u >0$.

\bigskip

{\flushleft \bm{{\bf The Kernel Condition $u \in \text{Ker}(\partial^*) \cap \text{Ker}(\bar{\partial}^*)$:}} }
\
\bigskip

Let $u \in C^{\infty}_{m-1,m-1}(M, \mathbb R)$ be a real form in $\mbox{Ker}(\partial^*)\cap \text{Ker}(\bar{\partial}^*)$. Write the Lefschetz decomposition of $u$ as

\begin{equation}
    \begin{aligned}
        u = f \omega^{m-1} + \psi_1 \wedge \omega^{m-2} + \sum\limits_{k =2}^{m-1}\psi_{k} \wedge \omega^{m-1-k}  
    \end{aligned}
\end{equation}

\noindent where $f \in C^{\infty}(M)$ and $\psi_{k} \in C^{\infty}_{k,k}(M)$ are the primitive coordinates of $u$ that satisfy $\Lambda(\psi_k) = 0$, or equivalently $\psi_{k} \wedge \omega^{n-2k+1} = 0$ for each $1 \leq k \leq m-1$. Then the Kernel condition becomes a system of first-order equations for primitives as follows. K\"ahler identity $[L, \sbpartial] = -i \partial $ implies the formula $\sbpartial (v \wedge \omega^p) = \sbpartial v \wedge \omega^{p}+ p \;i \partial v \wedge \omega^{p-1}$, so that

\begin{equation}\label{del-primitive}
    \begin{aligned}
        \sbpartial u = &\sbpartial f \wedge \omega^{m-1} + ( i(m-1)\partial f + \sbpartial \psi_1) \wedge \omega^{m-2} \\
        &+ \sum_{k=2}^{m-1}((m-k)i\partial \psi_{k-1} + \sbpartial \psi_{k} )\wedge\omega^{m-k-1}.
    \end{aligned}
\end{equation}

This is not a primitive decomposition since $\partial \psi_k$ is not a primitive, although $\sbpartial \psi_{k}$ is. Let

\begin{equation}\label{prim-decomp2}
\partial \psi_k = (\partial \psi_k)_{\text{prim}} + \eta_k \wedge \omega 
\end{equation}

\

\noindent be the primitive decomposition of $\partial \psi_k$, for $\eta_k\in C^{\infty}_{k,k-1}(M)$ a primitive form. The fact that $\eta_k$ is primitive follows from the following. The commutation relation $[\Lambda, \partial] = i\bar{\partial}^*$ implies $\Lambda(\partial \psi_k) = i \sbpartial \psi_k$. So $\Lambda(\partial \psi_{k})$ is primitive. If 

$$\eta_k = \eta_k^0 + \eta_k^1 \wedge \omega + \dots $$

\noindent is a primitive decomposition, then applying $\Lambda$ to \eqref{prim-decomp2} shows

\begin{equation}
    \begin{aligned}
    \Lambda(\partial\psi_k) = c_0 \eta_k^0 + c_1 \eta_k^1 \wedge \omega + \dots
    \end{aligned}
\end{equation}

\noindent for non-zero constants $c_i$. This means $\eta_k^i=0$ for $i\geq 2$ and $\eta^0_k$ is primitive.





\

We also have

\begin{equation}
    \begin{aligned}
        i \sbpartial \psi_k &= \Lambda(\partial \psi_k) \\
        &= \Lambda(\eta_k \wedge \omega) = (n-2k+1) \eta_k.
    \end{aligned}
\end{equation}

Hence

\begin{equation}
    \begin{aligned}
        (\partial \psi_k)_{\text{prim}} = \partial \psi_k - \frac{i \sbpartial \psi_k \wedge \omega}{n-2k+1}.
    \end{aligned}
\end{equation}

\

Plug this into \eqref{del-primitive} and using $\sbpartial f  =0$ we get the primitive decomposition of $\sbpartial u$.

\begin{equation}
    \begin{aligned}
        \sbpartial u &= (\sbpartial \psi_1 + i(m-1)\partial f + (m-2)i \eta_1) \wedge \omega^{m-2} \\
        &\;\;+ \sum_{k=2}^{m-1}\left[(m-k)i(\partial \psi_{k-1})_{prim} + \sbpartial \psi_{k} + (m-k-1)i\eta_{k}\right]\wedge\omega^{m-k-1}\\
        &=\left[\frac{n-m+1}{n-1}\sbpartial \psi_1 + i(m-1)\partial f \right] \wedge \omega^{m-2}+ \sum_{k=2}^{m-1}\Big[ (m-k)i\partial \psi_{k-1} \\
        &\;\;+ \frac{m-k}{n-2k+ 3} \sbpartial \psi_{k-1} \wedge \omega + \frac{n-m-k+2}{n-2k+1}\sbpartial \psi_{k} \Big]\wedge\omega^{m-k-1}
    \end{aligned}
\end{equation}

The Kernel condition $\sbpartial u =0 $ now gives the following general identity relating primitives of $u$.

\begin{equation}\label{cascade}
    (m-k)i\partial \psi_{k-1} + \frac{m-k}{n-2k+ 3} \sbpartial \psi_{k-1} \wedge \omega + \frac{n-m-k+2}{n-2k+1}\sbpartial \psi_{k} = 0
\end{equation}

\

\noindent for each $1 \leq k \leq m-1 $, assuming $\psi
_0 = f$. For $k=1$, this gives

\begin{equation}\label{primitives12}
    \sbpartial\psi_{1} = -\frac{i(m-1)(n-1)}{n-m+1} \partial f.
\end{equation}

\ 

We note the following implication which will not be used further. 

\begin{lemma}
    If $u$ satisfies $\sbpartial u =0$, then $\partial \sbpartial \psi_{k} = 0$ for all $0 \leq k \leq m-2$ and $k \neq \dfrac{n}{2}$. 

    \end{lemma}

\begin{proof}
    Applying $\sbpartial$ to \eqref{cascade}, we get

\begin{equation}
    \begin{aligned}
        (m-k)i \sbpartial \partial \psi_{k-1} + \frac{(m-k)}{(n-2k+3)} i \partial \sbpartial \psi_{k-1} = 0
    \end{aligned}
\end{equation}

By the anti-commutation $\partial \sbpartial = -\sbpartial \partial$, we get the desired result.

\end{proof}

We recall the Dolbeault Green's operator $G_\partial : C^\infty_{p,q}(M) \to C^\infty_{p,q}(M)$ which is the unique bounded linear operator that commutes with $\Delta_\partial$ and satisfies $I = \mathcal{H}_\partial + \Delta_\partial G_\partial$, where $\mathcal{H}_\partial$ is the orthogonal projection onto the space of harmonic forms. In addition, $G_{\partial}$ satisfies

\

\begin{enumerate}
    \item $G_{\partial}(\alpha) = 0$ for harmonic forms $\alpha$.
    \item $G_{\partial}(\alpha)$ gives the unique solution to $\Delta_{\partial} \beta = \alpha$, that is orthogonal to harmonic forms assuming $\alpha$ is orthogonal to harmonic forms. 

    \item $G_{\partial}$ commutes with $\partial$ and $\spartial$.
\end{enumerate}

\

It is known that if $\partial \alpha = \beta$, $\alpha$ must have the form

\begin{equation}
    \alpha = G_{\partial} \partial^* \beta + \gamma.
\end{equation}

\noindent for any $\partial$-closed form $\gamma$. In fact, this is an immediate consequence of the Hodge theorem which on compact complex manifolds gives the decomposition

$$\alpha = \partial (G_{\partial} \spartial \alpha) + \spartial(G_{\partial} \partial \alpha) + \alpha_{H},$$

\

\noindent for $\alpha_H$ a harmonic form \cite[Chapter $0$]{Griffiths-Harris}. If $\alpha$ is real and $\sbpartial$-exact, the solution is unique with $\gamma =0$. From \eqref{cascade} $\psi_{k}$'s must have the form

\begin{equation}\label{dbar-problem}
    \begin{aligned}
         \psi_{k-1} = G_{\partial} \partial^*\left( \frac{i}{n-2k+ 3} \sbpartial \psi_{k-1} \wedge \omega + \frac{i(n-m-k+2)}{(m-k)(n-2k+1)}\sbpartial \psi_{k} \right) + \gamma
    \end{aligned}
\end{equation}

\noindent for all $2 \leq k \leq m-1$.

\

\begin{claim}\label{harmonicity-claim}
    In the above setting if $\psi_k$ is $\sbpartial$-exact for all $1 \leq k \leq m-2 $ and $\sbpartial \psi_{m-1} = 0$, then all the intermediate primitives $\psi_{k}$ for $1 \leq k \leq m-2$ are zero, and $f$ is a constant. If $\psi_k$'s are assumed to be only $\sbpartial$-closed, then all primitives up to $k=m-2$ are harmonic.
\end{claim}

\








\begin{proof}

From \eqref{dbar-problem}, it follows that 

\begin{equation}\label{claim1.1-1}
    \psi_{m-2} = G_{\partial} \partial^*\left( \frac{i}{n-2m+ 5} \sbpartial \psi_{m-2} \wedge \omega \right)
\end{equation}

Let $D_k$ be the differential operator 

$$D_k \alpha =\frac{i}{n-2k+3} G_{\partial} \partial^*\left(  L \sbpartial \alpha \right).$$

\ 

Then we show that $D_k(\alpha) = \alpha$ can have only the trivial real-solution. Let $\alpha$ be such a solution. Since $G_{\partial}$ commutes with $\partial^*$, $\alpha$ is $\partial^*$-exact. This means $\alpha$ is orthogonal to harmonic forms by the Hodge theorem. By K\"ahler commutation relations

\begin{equation*}
\begin{aligned}
       D_k(\alpha) &= \frac{i}{n-2k+3} G_{\partial}(L \partial^* \sbpartial \alpha - i \bpartial \sbpartial \alpha) \\
       &= \frac{i}{n-2k+3} \left(-G_{\partial}(L \sbpartial \partial^* \alpha) - i \bpartial G_{\partial}(\sbpartial \alpha)\right) = \frac{1}{n-2k+3}  \bpartial G_{\partial}(\sbpartial \alpha)
\end{aligned}
\end{equation*}

\noindent since $\alpha$ is $\partial^*$-exact and $G_{\partial}$ commutes with $\bpartial$. It follows that $\alpha$ must be $\bpartial$-exact, and

$$\bpartial \alpha = \bpartial D_k(\alpha) = 0.$$

But since $\alpha$ is real, this implies $\partial \alpha =0$. So 

$$\Delta \alpha = (\partial \partial^* +\partial^* \partial) \alpha = 0. $$

Hence it follows that $\alpha = 0$, since it is also orthogonal to harmonic forms. 

\

Using this in \eqref{claim1.1-1} and $\sbpartial \psi_{m-1} =0$, we infer that $\psi_{m-2} = 0$. By induction, it follows that $\psi_{k} =0$ for all $1 \leq  k \leq m-2$. Clearly, if $\psi_1 = 0$, $f$ is a constant by \eqref{primitives12}. The second part of the Claim follows directly from  \eqref{cascade}.








\end{proof}

\






\begin{corollary}\label{lemma3}
    Let $u^1$ and $u^2$ be two real $(m-1,m-1)$ forms in $\Ker (\spartial) \cap \Ker( \bar{\partial}^*)$. Then assuming $\psi^1_{k} -\psi^2_{k}$ is $\sbpartial$-exact for all $1 \leq k \leq m-2$ and $\psi^1_{m-1} -\psi^2_{m-1}$ is $\sbpartial$-closed 
    we have 
    
    $$u^1(z) = u^2(z) + c \omega^{m-1} + (\psi^1_{m-1} -\psi^2_{m-1}).$$
    
\end{corollary}

\

Define

\begin{equation}
Q(u) = \star_{\omega}(i \pbp u \wedge \omega_{n-m-1}).
\end{equation}

\

Then \cite[Cor.~3.4]{DP25} shows that

\begin{equation}\label{defineQ}
        Q(u)= \frac{1}{(m-1)!}\left[-i \pbp \Lambda^{m-1} u + (m-1) \Delta \Lambda^{m-2} u - (\Delta \Lambda^{m-1} u) \;\omega\right].
\end{equation}

\

So the Monge-Amp\`ere equation \eqref{MA-eqn} can now be written as 

\begin{equation}\label{QMA-eqn}
    \left[\star_{\omega}(\alpha\wedge \omega_{n-m-1}) + Q(u)\right]^n = dV.
\end{equation}

\

By taking the Trace of $u$ repeatedly it can be shown that 

\begin{equation}\label{trace-2}
    \Lambda^{m-2} u =\frac{(m-2)!(n-2)!}{(n-m)!} \psi_1 + \frac{(m-1)!(n-1)!}{(n-m+1)!} f \; \omega,
\end{equation}

\noindent and 

\begin{equation}\label{trace-1}
\Lambda^{m-1} u = \frac{(m-1)! n!}{(n-m+1)!}f.
\end{equation}

Using \cref{defineQ,trace-1,trace-2}, equation \eqref{QMA-eqn} can be written as

\begin{equation}\label{QMA-primitives}
    \Big[\star_{\omega}(\alpha\wedge \omega_{n-m-1}) + \frac{(n-2)!}{(n-m)!}\left(\Delta \psi_1 - c_{n,m}\left[(n-m+1)\Delta f \omega +  n i \pbp f \right]\right)\Big]^n = dV.
\end{equation}

\noindent where $c_{n,m} =\dfrac{(n-1)}{(n-m+1)}$.

\

We prove a uniqueness theorem for \eqref{MA-eqn}.

\begin{theorem}\label{unique-thm}
Let $u^1$ and $u^2$ in $C_{m-1,m-1}^{\infty}(M, \mathbb R)$ be two solutions to \eqref{MA-eqn} that lie in $\Ker (\spartial) \cap \Ker( \bar{\partial}^*)$. Then

\begin{enumerate}
    \item If the difference of all primitives of orders from $2$ to $(m-2)$ of $u^1$ and $u^2$ is $\bar\partial^*$-exact and $\sbpartial(\psi^1_{m-1} - \psi_{m-1}^2) =0 $, then

$$u^1(z) = u^2(z) + c \omega^{m-1} + \alpha \wedge \omega^{m-2} +  (\psi^1_{m-1} - \psi_{m-1}^2).$$

\noindent for some constant $c$, a harmonic $(1,1)$ form $\alpha$.

\

\item If the difference of all primitives of orders from $2$ to $(m-1)$ of $u^1$ and $u^2$ is $\bar\partial^*$-closed, then

$$u^1(z) = u^2(z) + c \omega^{m-1} + \sum_{k=1}^{m-2} \alpha_k \wedge \omega^{m-k-1} + (\psi^1_{m-1} - \psi_{m-1}^2).$$

\noindent for some constant $c$ and harmonic forms $\alpha_k \in C^{\infty}_{k,k}(M, \mathbb R) $.
\end{enumerate}

\end{theorem}

\

\begin{proof}

We write $\Delta \psi_1$ in terms of $f$ and $\psi_2$ as follows. First apply $\spartial$ to \eqref{cascade} with $k=2$ to get (assuming $m\neq 2$)

\begin{equation}
    \begin{aligned}
        \spartial \partial \psi_{1} = \frac{1}{(n-1)} \left(\bpartial \sbpartial \psi_1 + i \spartial \sbpartial \psi_1 \; \omega\right) + i \frac{n-m}{(n-3)(m-2)} \spartial \sbpartial \psi_2 
    \end{aligned}
\end{equation}

It follows from \eqref{primitives12} that 

$$ \bpartial \sbpartial \psi_1 = -\frac{(m-1)(n-1)}{n-m+1} i\bpartial \partial f$$

\noindent and

$$i \spartial \sbpartial \psi_1 \; \omega = \frac{(m-1)(n-1)}{n-m+1} \Delta f \; \omega.$$

Combining the above equations we get

\begin{equation}
\begin{aligned}
    \Delta \psi_1 &= \partial \spartial \psi_1 + \spartial \partial \psi_1\\
    &= \frac{(m-1)n}{n-m+1} i \partial \bpartial f + \frac{m-1}{n-m+1} \Delta f \; \omega + \frac{i(n-m)}{(n-3)(m-2)} \spartial \sbpartial \psi_2.
\end{aligned}
\end{equation}

Plugging this into \eqref{defineQ} using \eqref{trace-2} and \eqref{trace-1} gives the following.

\begin{equation}\label{Q-fp1}
    Q(u) = \frac{(n-2)!}{(n-m)!} \left[ \frac{n(m-n)}{n-m+1} (\Delta f \omega +   i \partial \bar{\partial} f) + i \frac{n-m}{(m-2)(n-3)} \partial^* \bar{\partial}^* \psi_2 \right]
\end{equation}

For $m=2$, the term $\partial^* \bar{\partial}^* \psi_2$ will be absent in the above formula. Now the PDE is expressed in terms of $f$ and the primitive $\psi_2$. If $u_1$ and $u_2$ are two solutions, then $\Lambda_{\rho}(Q(u_1 - u_2)) = 0$ with $\rho >0$ is a $(1,1)$ form defined by

$$\rho^{n-1} = \sum_{s=1}^{n} \left[ \star_{\omega} \Big( (\alpha + i\partial\bar{\partial}u_1) \wedge \omega_{n-m-1} \Big) \right]^{n-s} \wedge \left[ \star_{\omega} \Big( (\alpha + i\partial\bar{\partial}u_2) \wedge \omega_{n-m-1} \Big) \right]^{s-1},$$

\noindent which follows from the binomial expansion of

\begin{equation}
    \begin{aligned}
         (\star_{\omega}\left((\alpha + i \pbp u_1)\wedge \omega_{n-m-1}\right))^n -  (\star_{\omega}\left((\alpha + i \pbp u_2)\wedge \omega_{n-m-1}\right))^n = 0.
    \end{aligned}
\end{equation}

 By \eqref{Q-fp1} this implies that

\begin{equation}\label{max-principle}
    \begin{aligned}
    0 &=\Lambda_{\rho}\left(\frac{n(m-n)}{n-m+1} \Delta (f_1 -f_2) \; \omega +   \frac{n(m-n)}{n-m+1} i \partial \bar{\partial} (f_1 - f_2)\right),
    \end{aligned}
\end{equation}

\noindent since $\psi^1_{2}-\psi^2_{2}$ is $\bpartial^*$-closed by assumption. Define the operator

$$\mathcal L v = \frac{n(m-n)}{n-m+1}\left( \mbox{Tr}_{\rho}(\omega)  \; \omega^{i \bj}  +  \rho^{i \bj}\right)v_{i \bj}.$$

This is an elliptic operator. By \eqref{max-principle} $\mathcal L(f_1 - f_2) = 0$ on a compact manifold. Hence the maximum principle implies $f_1 - f_2 = \text{constant}$.

Now Theorem \ref{unique-thm} follows from the proof of Claim \eqref{harmonicity-claim} combined with the fact that if $f_1 - f_2$ is constant, then $\spartial(\psi^1_1  - \psi^2_1) = 0$. From equation \eqref{cascade}, we also get that $\partial(\psi^1_1  - \psi^2_1) = 0$. Hence $\psi^1_1  - \psi^2_1$ is harmonic.

\end{proof}






{\flushleft{\bf The Laplacian Trace condition:} }

\bigskip

In addition to above, a third condition is introduced in \cite{DP25}:

\

\begin{equation}\label{MA-eqn2}
\begin{aligned}
    &\left[\star_{\omega}\left((\alpha + i \pbp u)\wedge \omega_{n-m-1}\right)\right]^n = dV\\
\end{aligned}
\end{equation}

\noindent with $u$ satisfying 

\begin{equation}\label{generalIC}
    \alpha + i \pbp u >0, \quad u \in \Ker (\spartial) \cap \Ker( \bar{\partial}^*), \quad \text{ and }\Lambda^{m-2} \Delta u =0.
\end{equation}

\

\ This third condition is too restrictive since along with \eqref{trace-2} it implies

\begin{equation}
    \frac{(m-2)!(n-2)!}{(n-m)!} \Delta \psi_1 + \frac{(m-1)!(n-1)!}{(n-m+1)!} \Delta f \omega = 0.
\end{equation}

\

Here we used that $\Delta$ commutes with $\Lambda$. This commutation also implies that $\Delta \psi_1$ and $\Delta f$ are both primitives. So by the uniqueness of Lefschetz decomposition, $\Delta \psi_1 = \Delta f =0$. This means that $f$ is constant. So now equation \eqref{QMA-primitives} reads

$$ [\star_{\omega}(\alpha\wedge \omega_{n-m-1})]^n = dV,$$

\

\noindent which cannot hold for arbitrary $dV$. In view of this, we introduce an alternate gauge-fixing in the next section.











\section{$(a,b)$ Monge-Amp\`ere-type equation}\label{section3}

\bigskip

 Fix any two real constants $(a,b)$ satisfying $a+ nb = m-n-2$. Consider the equation

\begin{equation} \label{(n-1)-MAeqn}
    \begin{aligned}
        \left[\star_{\omega}(\alpha + i\pbp u) \wedge \omega_{n-m-1}\right]^n = dV
    \end{aligned}
\end{equation}

\

\noindent with conditions $\alpha + i\pbp u >0$, $u \in \Ker (\spartial) \cap \Ker( \bar{\partial}^*)$, and 

\

\begin{equation}\label{alt-third-condition}
     (m-1)\Lambda^{m-2} \Delta u = (1+a) i \pbp \Lambda^{m-1}  u +  (1+b) \Lambda^{m-1} \Delta u \;\omega.
\end{equation}

\

Taking $\Lambda$ on both sides of \eqref{alt-third-condition} shows that $a+nb = m-n-2$ is necessary. Under \eqref{alt-third-condition}, equation \eqref{(n-1)-MAeqn} simplifies to 

\begin{equation}\label{ab-MAeqn}
    \left(\star_{\omega}(\alpha \wedge \omega_{n-m-1}) + \frac{n!}{(n-m+1)!} \left[ a i\pbp f + b \Delta f \omega\right] \right)^n = dV.
    \end{equation}

Indeed, $Q(u)$ now becomes

\begin{equation}
    \begin{aligned}
        Q(u) &= \frac{1}{(m-1)!}\left(a  i \pbp \Lambda^{m-1}  u +  b \Lambda^{m-1} \Delta u \;\omega\right)\\
    \end{aligned}
\end{equation}

 Now using \eqref{trace-1} we get \eqref{ab-MAeqn}. We will call \eqref{ab-MAeqn} the {\em $(a,b)$ Monge-Amp\`ere-type equation}.

\

We show that the application to Demailly's inequality holds, using the $(a,b)$ Monge-Amp\`ere-type equation for $b=0$ and $a= m-n-2$. We adapt the analytic technique developed in \cite{DP25} and \cite{Popovici16}.

\begin{theorem}\label{demailly-thm}
    Let $m \in \{1, \hdots, n-1\}$. Suppose $\alpha, \beta \in C^{\infty}_{m,m}(M, \mathbb R)$ are such that

    \begin{enumerate}
        \item $d\alpha =0 $, $\alpha >0$ and $\Delta \alpha_{\omega} =0$.

        \ 

        \item $\beta \geq C \omega_m$ for some constant $C>0$ and $\pbp \beta =0$.

        \ 

        \item $\dfrac{1}{n!}\int_M \alpha_{\omega}^n - \dfrac{1}{(n-m)!}\int_M\alpha_{\omega}^{n-m}\wedge \beta >0$
    \end{enumerate}

    \

    Then there exists a real $(m,m)$ current $T$ such that 

    $$T \geq \delta (\alpha_{\omega})_m \text{ for some constant } \delta >0,\hspace{0.5cm} \pbp T = 0, \text{ and } \hspace{0.5cm} T \in [(\alpha_{\omega})_m - \beta]_A.$$

    If in addition, $d \beta = 0$, $T$ can be found such that $dT = 0$ and $T$ lies in $[(\alpha_{\omega})_m - \beta]_{BC}$. Here $[.]_{BC}$ stands for the Bott-Chern cohomology and $[.]_A$ stands for the Aeppli cohomology class.
\end{theorem}

\begin{proof}

By Lamari-type duality lemma \cite{Lamari99, DP25}, to prove that there exists a real current $S$ and a constant $\delta>0$ such that $T = (\alpha_{\omega})_{m} - \beta + i \pbp S$ satisfies the Theorem, it is enough to show that there exists a constant $\delta>0$ such that 

$$(1-\delta)\int_{M}(\alpha_{\omega})_{m}\wedge\Omega\geq\int_{M}\beta\wedge\Omega,$$

  \noindent for any $\Omega \in C^{\infty}_{n-m,n-m}(M, \mathbb R)$ satisfying $\pbp \Omega = 0$ and $\Omega>0$ weakly. Assume for contradiction that there is a sequence of forms $\Omega_k$ and $\delta_k \to 0$ such that

  \begin{equation}\label{contra-sequence}
       (1-\delta_k)\int_{M}(\alpha_{\omega})_{m}\wedge\Omega_k<\int_{M}\beta\wedge\Omega_k.
  \end{equation}
 
  By assumption $2$, it is possible to normalize the sequence $\Omega_k$ so that $\int_{M}\beta\wedge\Omega_k = 1$ for each $k$. This also implies $$\int_M \omega_m \wedge \Omega_k \le \frac{1}{C},$$

  \noindent so that $\Omega_k \to \Omega_{\infty}$ weakly. Taking limits of \eqref{contra-sequence} gives
  
  $$\lim_{k \to \infty} \int_{M} (\alpha_{\omega})_{m} \wedge \Omega_{k} = \int_{M} (\alpha_{\omega})_{m} \wedge \Omega_{\infty} \leq 1.$$
  
  Let $u_k$ be a solution of 

  \begin{equation}
      \frac{1}{n!} \left[ \star_\omega \Big( (\alpha + i\partial\bar{\partial}u_k) \wedge \omega_{n-m-1} \Big) \right]^n = c_k \beta \wedge \Omega_k
  \end{equation}

  \noindent such that 

  \begin{equation}
      \tilde{\alpha}_k = \alpha + i\partial\bar{\partial}u_k > 0 \quad \text{and} \quad u_k \in \Ker(\partial^*) \cap \Ker(\bar{\partial}^*),
  \end{equation}

  \noindent and

  \begin{equation}\label{generalIC1}
      (m-1)\Lambda^{m-2} \Delta u_k = (m-n-1)i \pbp \Lambda^{m-1}  u_k +  \Lambda^{m-1} \Delta u_k \;\omega.
  \end{equation}   

  \

  \noindent which exists by Theorem \ref{exist-theorem}.

  \

  By H\"older's inequality and the pointwise inequality $\dfrac{\rho_m \wedge \Omega}{\omega_n} \cdot \dfrac{\rho_{n-m} \wedge \beta}{\rho_n} \geq \dfrac{\beta \wedge \Omega}{\omega_n}$ for any $(1,1)$ form $\rho>0$ (See \cite[p.~25]{DP25})

$$\begin{aligned}
&\left( \int_M ((\tilde{\alpha}_k)_\omega)_m \wedge \Omega_k \right) \cdot \left( \int_M ((\tilde{\alpha}_k)_\omega)_{n-m} \wedge \beta \right) \geq c_k \left( \int_M \beta \wedge \Omega_k \right)^2 = c_k.
\end{aligned}$$

Note that $c_k = \int_M ((\tilde{\alpha}_k)_\omega)_n $. By \eqref{generalIC1}

$$\begin{aligned}
(\tilde{\alpha}_k)_\omega &= \star_{\omega}\left(\alpha \wedge \omega_{n-m-1}\right) + Q(u_k) \\
&= \alpha_\omega  + \frac{m-n-2}{(m-1)!}\left( i \pbp \Lambda^{m-1}  u_k \right).
\end{aligned}$$

\

It follows that for any $p \in \{1,\hdots, n\}$

$$((\tilde{\alpha}_k)_\omega)_p - (\alpha_\omega)_p = i\pbp \Lambda^{m-1} u_k \wedge \Gamma ,$$

\

\noindent for a $d$-closed form $\Gamma$. Note that $\alpha_{\omega}$ is harmonic and hence $d$-closed. Now applications of the Stokes' theorem using $\pbp \Omega_k=0$ and $\pbp \beta = 0$ gives

$$ \int_M ((\tilde{\alpha}_k)_\omega)_m \wedge \Omega_k = \int_M (\alpha_\omega)_m \wedge \Omega_k,$$

$$ \int_M ((\tilde{\alpha}_k)_\omega)_{n-m} \wedge \beta =  \int_M (\alpha_\omega)_{n-m} \wedge \beta , $$

\noindent and

$$\int_M ((\tilde{\alpha}_k)_\omega)_n =\int_M ({\alpha}_\omega)_n .$$

\

Combining the above equations with $\int_{M} (\alpha_{\omega})_{m} \wedge \Omega_{\infty} \leq 1$, we have

$$\int_{M} (\alpha_{\omega})_{n-m} \wedge \beta \geq \int_{M} (\alpha_{\omega})_{n}$$

\

\noindent contradicting the assumption $3$.

\end{proof}

\section{\emph{ A priori} estimates for $(a,b)$ Monge-Amp\`ere-type equation}

\bigskip

In this section, we give a complete description of the parametric regimes for $(a,b)$ in which the $(a,b)$ Monge-Amp\`ere-type equation is either solvable, or conditionally solvable based on the background data. Also refer to \cite{GQY19,STW17,TW17} for fully-nonlinear equations of similar type. 

\begin{itemize}
    \item {\bf Regime $1$}: $a+b\geq 0$ and $b \geq 0$, \textbf{or}  $a+b \leq 0$ and $b \leq 0$.

In this range, the equation \eqref{ab-MAeqn} is unconditionally solvable. It is enough to consider the case when $a+b\geq 0$ and $b \geq 0$, since the other region is covered by the substitution $f \to -f$.

For any $a,b$ such that $a+b >0$ and $b > 0$, we derive independent second-order and gradient estimates for \eqref{ab-MAeqn}. The case when $a+b=0$ is the Monge-Amp\`ere equation for $(n-1)$-plurisubharmonic functions which has been dealt with in \cite{George25, TW17}. The case when $b = 0$ reduces to the complex Monge-Amp\`ere equation which has been treated extensively in the literature. 

\

We consider the equation

\begin{equation}\label{abMA-eqn+}
    \begin{cases}
        &\log \det( X + a i \pbp f + b \Delta f \omega) = \psi,\\
        &X + a i \pbp f + b \Delta f \omega >0
    \end{cases}
\end{equation}

\noindent where $X$ is a smooth, positive $(1,1)$ form. Denote $\tilde g =  X + a i \pbp f + b \Delta f \omega $. The symbol of the differential operator is given in orthonormal coordinates for $\{g_{i \bj}\}$ by 

\begin{equation}\label{elliptic+}
    \begin{aligned}
        G^{i \bj} = a\tilde{g}^{i \bj} + b \delta_{i j} \sum_k \tilde g^{k \bk}.
    \end{aligned}
\end{equation}

\noindent where $\tilde g^{i \bj}$ denotes the inverse of $\tilde g_{i \bj}$.

\begin{theorem}
   Any solution $f \in C^4(M)$ with $|\nabla f|_{\omega} \leq C$ to equation \eqref{ab-MAeqn} satisfies

   $$|i\pbp f|_{\omega} \leq C,$$

   \noindent for a constant $C$ that depends only on the background data.
\end{theorem}

\begin{proof}

The largest eigenvalue $\theta_1$ of $i \pbp f$ can be estimated by applying maximum principle to the following test function \cite{TW17,Szekelyhidi18,STW17}. Let $P(f) = \log(\theta_1) + \varphi(|\nabla f|^2)$, where 

$$\varphi(t) = -\log\left(1 - \frac{t}{2(1+ \sup|\nabla f|^2)}\right).$$

\ 

Consider a local holomorphic coordinate system that diagonalizes the Hessian at the maximum point $z_0$ of $P(f)$ such that $g_{i \bj}(z_0) = \delta_{i j}$ and $dg_{i \bj}(z_0) = 0$. So at $z_0$, $f_{i \bj} = \theta_i \delta_{ij}$. We prove by contradiction that $\theta_1 \leq C \sup( 1+ |\nabla f|^2)$. Differentiating the PDE by $\nabla_1 \nabla_{\bar 1}$,

\begin{equation}
    \begin{aligned}
        \tilde g^{i \bj}  \nabla_{1} \nabla_{\bar 1} X_{i \bj} + G^{i \bj} \nabla_{1} \nabla_{\bar 1} f_{i \bj} - \tilde g^{i \bl}\tilde g^{k \bj} \nabla_{\bar 1} \tilde{g}_{i \bj} \nabla_1 \tilde g_{k \bl} = \psi_{1 \bar 1},
    \end{aligned}
\end{equation}

\

\noindent and using the curvature formula

$$\nabla_{1} \nabla_{\bar 1} f_{i \bj} =\nabla_i \nabla_{\bar j} f_{1 \bar 1} + \sum_p (R_{1 \bar 1 i \bar p} f_{p \bar j} - R_{1 \bar j i \bar p} f_{p \bar 1}),$$

\noindent gives 

\begin{equation}
    \begin{aligned}
          G^{i \bj}\nabla_i \nabla_{\bar j} f_{1 \bar 1} =& \; \psi_{1 \bar 1}-G^{i \bj} \sum_p (R_{1 \bar 1 i \bar p} f_{p \bar j} - R_{1 \bar j i \bar p} f_{p \bar 1})  + \tilde g^{i \bl}\tilde g^{k \bj} \nabla_{\bar 1} \tilde g_{i \bj} \nabla_1 \tilde g_{k \bl}\\
          & -  \tilde g^{i \bj}  \nabla_{1} \nabla_{\bar 1} X_{i \bj}.
    \end{aligned}
\end{equation}

The eigenvalue derivative formulae are given by

    $$\nabla_i \theta_1 = \nabla_i f_{1\bar{1}},$$

   $$\nabla_i \nabla_{\bar{j}} \theta_1 = \nabla_i \nabla_{\bar{j}} f_{1\bar{1}} + \sum_{p > 1} \frac{\nabla_i f_{1\bar{p}} \nabla_{\bar{j}} f_{p\bar{1}}}{\theta_1 - \theta_p}.$$

So we have

\begin{equation}
    \begin{aligned}
        G^{i \bj} \nabla_i \nabla_{\bj} P(f) &\geq \frac{1}{\theta_1} \left(\psi_{1 \bar 1}-G^{i \bj} \sum_p (R_{1 \bar 1 i \bar p} f_{p \bar j} - R_{1 \bar j i \bar p} f_{p \bar 1}) + \tilde g^{i \bl}\tilde g^{k \bj} \nabla_{\bar 1} \tilde g_{i \bj} \nabla_1 \tilde g_{k \bl}\right) \\
        & \;\;\;- G^{i \bj} \frac{\nabla_{i} \theta_{1} \nabla_{\bj} \theta_{1}}{\theta_1^2} + \varphi' G^{i \bi} \theta_i^2 + \varphi' G^{i \bj} \sum_k\nabla_i f_k \nabla_{\bj} f_{\bk} \\
        & \;\;\;+ \varphi''G^{i \bj} \nabla_i|\nabla f|^2 \nabla_{\bj} |\nabla f|^2 -   \frac{1}{\theta_1}\tilde g^{i \bj}  \nabla_{1} \nabla_{\bar 1} X_{i \bj}+ \sum_{p > 1} G^{i \bj}\frac{\nabla_i f_{1\bar{p}} \nabla_{\bar{j}} f_{p\bar{1}}}{\theta_1(\theta_1 - \theta_p)} \\
        &\;\;\;- C|\nabla f|(1 + \sum \tilde g^{i \bi})\\
        &\geq -C \sum_i \tilde g^{i \bi} + \min\{(a+b),b\}\;\varphi'\sum_k \tilde g^{k \bk} \theta_1^2 \\
        &\;\;\; + (\varphi'' - (\varphi')^2)G^{i \bj} \nabla_i|\nabla f|^2 \nabla_{\bj} |\nabla f|^2 \\
        &>0.
    \end{aligned}
\end{equation}

In the final inequality we used $\varphi'(|\nabla f|^2) \geq \dfrac{1}{\theta_1}$, 

\

\begin{equation}\label{Gg}
    G^{1 \bar 1} \geq \min\{(a+b),b\} \sum_k \tilde g^{k \bk} ,
\end{equation}

\noindent and $\varphi''= (\varphi')^2$. This contradicts the maximum principle that $G^{i \bj} \nabla_i \nabla_{\bj} P(f) \leq 0$ at the maximum point $z_0$ of $P(f)$. 

\end{proof}

We derive direct gradient estimates for the solution $f$. Define the norms 

$$|\nabla v|_{G}^2 = G^{i \bj} v_{i} \ol{v_{j}},$$

$$|\nabla \nabla v|_{Gg}^2 = g^{i \bj}G^{k \bl} v_{ik} \ol{v_{lj}}, \quad |\nabla \bnabla v|_{Gg}^2 = g^{i \bj}G^{k \bl} v_{i\bl} \ol{v_{j \bk}}$$

\noindent where the indices here stand for covariant derivatives with respect to the background metric $g$. 

\

\begin{lemma}\label{lemma2}

Let $A>0$ be a constant.
 In normal coordinates at a point

    \begin{equation*}
        \begin{aligned}
            G^{i \bj} \nabla_i \nabla_{\bj} (|\nabla f|^2 + A f^2) \geq \;  & 2 f A(n- \tilde g^{i \bj} X_{i\bj}) +2 A\min\{a+b, b\}\sum_k \tilde g^{k \bk}  |\nabla f|^2 \\
            &- C|\nabla f|(1+\sum_i \tilde g^{i \bi})+G^{i \bj }R_{k \bj i \bp} f_p f_{\bk} .
        \end{aligned}
    \end{equation*}
    
\end{lemma}

\begin{proof}

By differentiating the PDE by $z_k$ and $z_{\bk}$, using the commutation formulae

\

$$
\begin{aligned}
   &[\nabla_k, \nabla_{\bj}] f_i = -R_{k \bj i \bp} f_p,\\
   &[\nabla_\bk, \nabla_{\bj}] f_i = 0, 
\end{aligned}
$$

\noindent we get

\begin{equation}\label{g-1}
    \begin{aligned}
        &G^{i \bj} f_{k i \bj} = \nabla_k \psi - \tilde g^{i \bj}  \nabla_{k} X_{i \bj} + G^{i \bj }R_{k \bj i \bp} f_p,\\
        &G^{i \bj}  f_{\bk i \bj} = \nabla_{\bk} \psi - \tilde g^{i \bj} \nabla_{\bk} X_{i \bj}.
    \end{aligned}
\end{equation}

Hence

\begin{equation}\label{g-2}
    \begin{aligned}
        G^{i \bj} \nabla_i \nabla_{\bj} (|\nabla f|^2 + A f^2) &\geq |\nabla \nabla f| _{Gg}^2 + |\nabla \bnabla f|_{Gg}^2 + 2 Af \; G^{i \bj} f_{i \bj}\\
        &\quad +2 A |\nabla f|_{G}^2 + f_k G^{i \bj} f_{\bk i \bj} + f_{\bk} G^{i \bj} f_{k i \bj}.\\
    \end{aligned}
\end{equation}

Observe that 

\begin{equation}\label{g-3}
\begin{aligned}
        G^{i \bj} f_{i \bj} &= \tilde g^{i \bj} (a f_{i \bj} + b \delta_{i j} \sum_k f_{k \bk})\\
    & = \tilde g^{i \bj}(\tilde g_{i \bj} - X_{i \bj}) = n- \tilde g^{i \bj} X_{i\bj}.
\end{aligned}
\end{equation}

So now \eqref{g-2} becomes

\begin{equation}
\begin{aligned}
     G^{i \bj} \nabla_i \nabla_{\bj} (|\nabla f|^2 + A f^2) &\geq |\nabla \nabla f| _{Gg}^2 + |\nabla \bnabla f|_{Gg}^2 + 2 f A(n- \tilde g^{i \bj} X_{i\bj})\\
        &\quad +2 A |\nabla f|_{G}^2 - C|\nabla f|(1+\sum_i \tilde g^{i \bi}) + G^{i \bj }R_{k \bj i \bp} f_p f_{\bk}. \\
\end{aligned}
\end{equation}

By \eqref{elliptic+}

$$|\nabla f|_{G}^2\; \geq \min\{a+b,b\} \sum_k \tilde g^{k \bk}|\nabla f|^2.$$

This gives the desired inequality.

\end{proof}

At the point of maximum of $|\nabla f|^2 + A f^2$, it follows from Lemma \ref{lemma2} and \eqref{Gg} that 

\begin{equation}
    \begin{aligned}
        0 \geq \sum_k \tilde g^{k \bk} (2A\min\{a+b, b\} -C)|\nabla f|^2 - CA(1+\sum_i \tilde g^{i \bi})|\nabla f|
    \end{aligned}
\end{equation}

Note that $\sum\limits_k \tilde g^{k \bk} \geq n e^{-\frac{\psi}{n}}$ by the AM-GM inequality. Choosing $A$ larger than $\dfrac{C}{2\min\{a+b, b\}}$ now gives the gradient estimates $|\nabla f| \; \leq C$. 

\begin{theorem}\label{ABPexistence}
    Assuming $X >0$, there exists a smooth solution to \eqref{abMA-eqn+} unique up to an additive constant.
\end{theorem}

   The $C^0$ estimates for $f$ can be obtained using the ABP maximum principle as in \cite{Szekelyhidi18}. Note that strong positivity implies the positive-definiteness of the coefficient matrix, which is sufficient for the ABP maximum principle. We omit the argument here. Now by Evans-Krylov theory and the Schauder theorem, $|f|_{C^{\infty}}$ is uniformly bounded and Theorem \ref{ABPexistence} follows from standard continuity method arguments.

\

\item {\bf Regime $2$: $a+b < 0$, $b > 0$ or $a+b >0$, $b < 0$.}

\

 Let $\kappa(A) = \dfrac{\lambda_{max}(A)}{\lambda_{min}(A)}$ denote the condition number of a matrix $A$, the ratio of the largest to the smallest eigenvalue. Without loss of generality, we assume that $a+b>0$ and $b<0$. Otherwise replace $f$ by $-f$. To prove part $2$ of Theorem \ref{ab-MAeqn-theorem}, we assume $\kappa(X) < \dfrac{a}{n|b|}$.

 In orthonormal coordinates that diagonalize $\tilde g$ at a point, we write $\tilde g_{i \bj} = \lambda_i \delta_{i j}$ where $\lambda_1 \geq  \lambda_2 \geq \hdots \geq \lambda_n >0$.

 $$G^{i \bj} = (a\lambda_i^{-1} - |b|  \sum_{k} \lambda_k^{-1}) \delta_{i j}. $$

So $G^{i \bj} = \mu_i \delta_{i j}$ is diagonalized at this point with eigenvalues  $0 < \mu_1 \leq \hdots \leq \mu_n $. Clearly

\begin{equation}\label{smallest-ev}
    \mu_1 \geq a\lambda_1^{-1}\left(1 - \frac{n|b|}{a}\kappa(\tilde{g})\right).
\end{equation}

For uniform ellipticity along the continuity path, a positive lower bound on $\mu_1$ is necessary. This is automatic in Regime $1$ from \eqref{Gg}. But this might not hold in Regime $2$ because of the parameter constraints. Define the set of Hermitian metrics

$$\mathcal G_{\delta} = \{\tilde g >0 \; | \; \kappa(\tilde g) \leq \kappa(X) + \delta \},$$

\noindent where $0<\delta< \left(\dfrac{a}{n|b|} - \kappa(X)\right)$ is some constant. Clearly $X \in \mathcal G_{\delta}$.

\

Let $\tilde g_f = X + a i \pbp f + b \Delta f \;\omega$ and fix a $0 < \gamma <1$. Then we prove the following.

\begin{theorem}\label{QIFT}

There exists a small constant $\epsilon >0$ and an $R>0$ such that if
$$|\psi - \log\det(X)|_{C^{0, \gamma}} < \epsilon$$

\noindent  then there is a unique $f \in B_R(0) \subset C^{2, \gamma}_{0}$ such that $\tilde g_f \in \mathcal G_{\delta}$ and

$$\log\det(\tilde g_f) = \psi + c.$$

\noindent for some constant $c$. In addition,

$$|\kappa(\tilde g_f) - \kappa(X)|_{C^0(M)} \leq C_{\delta} \; | \log \det \tilde g_f -  \log \det X|_{C^{0,\gamma}(M)},$$ 

\noindent and
$$|f|_{C^{2, \gamma}} \leq 2\;C_1 |\psi - \log\det(X)|_{C^{0, \gamma}}$$

\noindent for uniform positive constants $C_1$ and $C_{\delta}$.

\end{theorem}

\begin{proof}

 This is a quantitative version of the implicit function theorem proved using the contraction mapping theorem. Denote the linearization of the PDE at $0$ by $\mathcal{L}$, so that

$$\mathcal{L}f = a\Delta_{X}f + b \;\mbox{tr}_{X}(\omega) \Delta_{\omega} f .$$

By the assumption on $X$, this is uniformly elliptic with kernel being constant functions. Hence by Fredholm theory, the kernel of the adjoint $L^*$ is $1$-dimensional. Let $\rho(x)$ span this kernel and $\int_M \rho \; \omega^n = 1$. Then define 
 
 $$C_0^{k,\gamma} = \{f \in C^{k,\gamma}(M) | \int_M f \rho \; \omega^n= 0  \},$$

  $$\Omega^{k, \gamma} = \{f \in C_0^{k, \gamma} |  \;\; \tilde g_f >0\},$$

 \noindent and 

 $$\Omega_{R'} = \{f \in C_0^{2,\gamma} | \;\;|f|_{C^{2,\gamma}} < R'\}$$

 \

 \noindent where $R'$ is defined below. Then $\mathcal{L}: C_0^{2, \gamma} \to C_0^{0, \gamma}$ has zero kernel, and is invertible in this domain by Fredholm alternative. Let $C_1= \|\mathcal{L}^{-1}\|$ be the operator norm. Define the projection $\pi: C^{k, \gamma} \to C_0^{k, \gamma} $ by

 $$\pi(f) = f- \int_M f \rho \omega^n .$$

 \

 
 
Note that $|\pi(f)|_{C^0}< 2|f|_{C^0}$. To write \eqref{abMA-eqn+} as a fixed-point equation, define $G(f) = \pi \circ \log \det( X + a i \pbp f + b \Delta f \omega)$ and the map $\Phi: \Omega_{R'} \to \Omega_{R'}$ by

$$\Phi(f) = f + \mathcal L ^{-1}\left[ \pi\circ\psi-( G(0) + \mathcal{L}(f) + \mathcal{R}(f) ) \right]$$

\noindent with

$$\mathcal{R}(f) = G(f) - G(0) - \mathcal L(f).$$

\

Both $\mathcal{R}$ and $G$ are defined from $\Omega^{2, \gamma} \to C_0^{0,\gamma}$. By choosing $\epsilon$ small we ensure that $\Phi$ maps to $\Omega_{R'}$. $\mathcal{R}(f)$ denotes the quadratic error in the Taylor approximation of $G(f)$ near $f=0$, and can be estimated as follows. 

$$|\mathcal R(f)|_{C^{0, \gamma}} \leq \int_0^1 (1-t)|D^2G(tf)(f,f)|_{C^{0, \gamma}} dt,$$

\noindent where the Fr\'echet derivative $D^2G(tf)(f,f)$ can be evaluated as

$$D^2G(tf)(f,f) = - \pi(\tilde g_{tf}^{i \bl} \tilde g_{tf}^{k \bj} (V_f)_{k \bl} (V_f)_{i \bj}),$$

\noindent for $V_f = a i \pbp f + b \Delta f \omega$. Let $R_0 = \dfrac{(a + n |b|)^{-1} \lambda_{min}(X)}{2}$. In the ball $B_{R_0}(0)$, $\tilde g_{tf}$ is uniformly equivalent to $X$, and consequently

$$|D^2G(tf)(f,f)|_{C^{0, \gamma}} \leq C'_2 |V_f|_{C^{0, \gamma}}^2 \leq C_2 |f|^2_{C^{2, \gamma}}.$$

So $|\mathcal R(f)|_{C^{0, \gamma}} \leq  C_2 |f|^2_{C^{2, \gamma}}.$ Similarly for any $h$,

$$ |D \mathcal R(f)h|_{C^{0, \gamma}} \leq 2 C_2 R' |h|_{C^{2, \gamma}}.$$

This is sufficient to show that $\Phi$ is a contraction mapping in the ball $B_{R'}(0)$, where

$$R' < \min\left\{R_0, \frac{1}{2 C_{1}C_{2}} \right\}.$$

Indeed, we have 

$$\Phi(f) = \mathcal L^{-1}[\tilde{\Psi} -\mathcal{R}(f)]$$

\noindent \noindent with $\tilde{\Psi} = \pi\circ\psi - G(0)$, and 

$$|\Phi (f_2) - \Phi (f_1)|_{C^{2, \gamma}} \leq C_1 |\mathcal{R}(f_2) - \mathcal{R}(f_1)|_{C^{0, \gamma}} \leq 2C_1 C_2 R'|f_2 - f_1|_{C^{2, \gamma}}.$$

By the choice of $R'$ this is a contraction. So by Banach contraction mapping theorem, there is a unique fixed point $f \in B_{R'}(0)$ for $\Phi$. Clearly the fixed-point solution of $\Phi$ satisfies $G(f) = \pi\circ\psi$. We also get that
\begin{equation}
    \begin{aligned}
        |f|_{C^{2, \gamma}} &= \; |\mathcal{L}^{-1}[\tilde \Psi - \mathcal R(f)]|_{C^{2, \gamma}}\\
        & \leq \; C_1 |\tilde \Psi|_{C^{0, \gamma}} + C_1C_2 |f|^2_{C^{2, \gamma}}\\
        &\leq C_1 |\tilde \Psi|_{C^{0, \gamma}} + \frac{1}{2} |f|_{C^{2, \gamma}}
    \end{aligned}
\end{equation}

Hence $|f|_{C^{2, \gamma}} \leq 2C_1|\tilde \Psi|_{C^{0, \gamma}}$. By local Lipschitz continuity of the condition number

 \begin{equation}
     \begin{aligned}
         |\kappa(\tilde g_f) - \kappa(X)|_{C^{0}} &\leq C_{\kappa} |\tilde g - X|_{C^0} \\
         &\leq C_{\kappa}(a+n |b|) |f|_{C^{2, \gamma}} \leq C'_{\delta} |\tilde \Psi|_{C^{0, \gamma}}\\
         & \leq 2C'_{\delta} | \log \det \tilde g_f -  \log \det X|_{C^{0,\gamma}(M)}.
     \end{aligned}
 \end{equation}

For $\epsilon$ small, this shows that $\tilde g_f \in \mathcal G_{\delta}$. Here $C_{\kappa}$ is the local Lipschitz constant of $\kappa$ near $X$.

Set $R = R'$ to complete the proof with $C_{\delta} = 2C'_{\delta}$.

\end{proof}

\

Theorem \ref{QIFT} implies that the operator remains uniformly elliptic in this range and gives $C^{2,\gamma}$ estimates for $f$ in $B_R(0)$. Now applying the Schauder theorem gives $C^{\infty}$ estimates for $f$. This completes the proof of Theorem \ref{ab-MAeqn-theorem}. 

\

\vspace{1cm}

{\bf Open Question:} It would be desirable to remove the small data assumption $|\psi - \log \det(X)|_{C^{0,\gamma}} < \epsilon$ in Regime $2$. But it is unclear how to ensure uniform ellipticity along the continuity path. It might be fruitful to consider other paths that can achieve this. The advantage in Regime $2$ is that $C^2$ estimates for the solution are immediate from the cone condition: $\mu_1>0$ implies that the ratio $\dfrac{\lambda_1}{\lambda_n}$ is bounded above. This combined with the PDE, uniformly bounds all $\lambda_i$. Possibly one could search for structural properties of the PDE that avoid the boundary of the elliptic cone, similar to how the Monge-Amp\`ere equation avoids the boundary hyperplanes $\{\lambda_i = 0\}$.

\end{itemize}

\clearpage

\bibliographystyle{plain}
\bibliography{references}

@article {BDPP13,
    AUTHOR = {Boucksom, S\'{e}bastien and Demailly, Jean-Pierre and P\u{a}un, Mihai and Peternell, Thomas},
     TITLE = {The pseudo-effective cone of a compact {K}\"{a}hler manifold and varieties of negative {K}odaira dimension},
   JOURNAL = {J. Algebraic Geom.},
  FJOURNAL = {Journal of Algebraic Geometry},
    VOLUME = {22},
      YEAR = {2013},
    NUMBER = {2},
     PAGES = {201--248},
      ISSN = {1056-3911},
   MRCLASS = {32J27 (14E30 32E30)},
  MRNUMBER = {3019449},
       DOI = {10.1090/S1056-3911-2012-00574-8},
       URL = {https://doi.org/10.1090/S1056-3911-2012-00574-8}
}

@inproceedings{Dem2010,
  title={Holomorphic Morse inequalities and asymptotic cohomology groups: a tribute to Bernhard Riemann},
  author={Demailly, Jean-Pierre},
  booktitle={Proceedings of the Riemann International School of Mathematics, Advances in Number Theory and Geometry},
  address={Verbania, Italy},
  month={April},
  year={2009},
  note={arXiv:1003.5067 [math.CV]}
}

@article{DP25,
AUTHOR = {Dinew, S\l awomir and Dan Popovici},
TITLE = {m-Pseudo-effectivity and a Monge-Ampère-Type Equation for Forms of Positive Degree},
JOURNAL = {arXiv:2510.27362},
YEAR = {2025}
}

@misc{George25,
      title={Complex Monge-Amp\`ere equation for positive $(p,p)$-forms on compact K\"ahler manifolds}, 
      author={Mathew George},
      year={2025},
      eprint={2411.06497},
      archivePrefix={arXiv},
      primaryClass={math.AP},
      url={https://arxiv.org/abs/2411.06497}, 
}

@book {Griffiths-Harris,
    AUTHOR = {Griffiths, Phillip and Harris, Joseph},
     TITLE = {Principles of algebraic geometry},
    SERIES = {Wiley Classics Library},
      NOTE = {Reprint of the 1978 original},
 PUBLISHER = {John Wiley \& Sons, Inc., New York},
      YEAR = {1994},
     PAGES = {xiv+813},
      ISBN = {0-471-05059-8},
   MRCLASS = {14-01},
  MRNUMBER = {1288523}
}

@article {GQY19,
    AUTHOR = {Guan, Bo and Qiu, Chunhui and Yuan, Rirong},
     TITLE = {Fully nonlinear elliptic equations for conformal deformations
              of {C}hern-{R}icci forms},
   JOURNAL = {Adv. Math.},
  FJOURNAL = {Advances in Mathematics},
    VOLUME = {343},
      YEAR = {2019},
     PAGES = {538--566},
      ISSN = {0001-8708,1090-2082},
   MRCLASS = {58J05 (35J60 53C55 58J32)},
  MRNUMBER = {3883214},
MRREVIEWER = {Mohammed\ El A\"{\i}di, Universidad Nacional de Colombia},
       DOI = {10.1016/j.aim.2018.11.008},
       URL = {https://doi.org/10.1016/j.aim.2018.11.008},
}

@article {Lamari99,
    AUTHOR = {Lamari, Ahc\`ene},
     TITLE = {Courants k\"ahl\'eriens et surfaces compactes},
   JOURNAL = {Ann. Inst. Fourier (Grenoble)},
  FJOURNAL = {Universit\'e{} de Grenoble. Annales de l'Institut Fourier},
    VOLUME = {49},
      YEAR = {1999},
    NUMBER = {1},
     PAGES = {vii, x, 263--285},
      ISSN = {0373-0956,1777-5310},
   MRCLASS = {32J27 (32C30 32J15 32Q15)},
  MRNUMBER = {1688140},
MRREVIEWER = {Thierry\ Bouche},
       DOI = {10.5802/aif.1673},
       URL = {https://doi.org/10.5802/aif.1673},
}

@article {Popovici16,
    AUTHOR = {Popovici, Dan},
     TITLE = {Sufficient bigness criterion for differences of two nef classes},
   JOURNAL = {Math. Ann.},
  FJOURNAL = {Mathematische Annalen},
    VOLUME = {364},
      YEAR = {2016},
    NUMBER = {1-2},
     PAGES = {649--655},
      ISSN = {0025-5831},
   MRCLASS = {32J27 (14F40)},
  MRNUMBER = {3451399},
       DOI = {10.1007/s00208-015-1224-2},
       URL = {https://doi.org/10.1007/s00208-015-1224-2}
}

@article {Szekelyhidi18,
    AUTHOR = {Sz\'{e}kelyhidi, G\'{a}bor},
     TITLE = {Fully non-linear elliptic equations on compact {H}ermitian
              manifolds},
   JOURNAL = {J. Differential Geom.},
  FJOURNAL = {Journal of Differential Geometry},
    VOLUME = {109},
      YEAR = {2018},
    NUMBER = {2},
     PAGES = {337--378},
      ISSN = {0022-040X,1945-743X},
   MRCLASS = {58J05 (32W50 35J60 35J96 35R01 53C55)},
  MRNUMBER = {3807322},
MRREVIEWER = {Bianca\ Santoro},
       DOI = {10.4310/jdg/1527040875},
       URL = {https://doi.org/10.4310/jdg/1527040875},
}

@article {STW17,
    AUTHOR = {Sz\'{e}kelyhidi, G\'{a}bor and Tosatti, Valentino and
              Weinkove, Ben},
     TITLE = {Gauduchon metrics with prescribed volume form},
   JOURNAL = {Acta Math.},
  FJOURNAL = {Acta Mathematica},
    VOLUME = {219},
      YEAR = {2017},
    NUMBER = {1},
     PAGES = {181--211},
      ISSN = {0001-5962,1871-2509},
   MRCLASS = {53C55 (32C36 32Q99)},
  MRNUMBER = {3765661},
MRREVIEWER = {Keizo\ Hasegawa},
       DOI = {10.4310/ACTA.2017.v219.n1.a6},
       URL = {https://doi.org/10.4310/ACTA.2017.v219.n1.a6},
}

@article {TW17,
    AUTHOR = {Tosatti, Valentino and Weinkove, Ben},
     TITLE = {The {M}onge-{A}mp\`ere equation for {$(n-1)$}-plurisubharmonic
              functions on a compact {K}\"{a}hler manifold},
   JOURNAL = {J. Amer. Math. Soc.},
  FJOURNAL = {Journal of the American Mathematical Society},
    VOLUME = {30},
      YEAR = {2017},
    NUMBER = {2},
     PAGES = {311--346},
      ISSN = {0894-0347,1088-6834},
   MRCLASS = {32W20 (32Q15 32U05 53C55)},
  MRNUMBER = {3600038},
MRREVIEWER = {Bianca\ Santoro},
       DOI = {10.1090/jams/875},
       URL = {https://doi.org/10.1090/jams/875},
}

@article {Xiao2015,
    AUTHOR = {Xiao, Jian},
     TITLE = {Weak transcendental holomorphic {M}orse inequalities on compact {K}\"ahler manifolds},
   JOURNAL = {Ann. Inst. Fourier (Grenoble)},
  FJOURNAL = {Annales de l'Institut Fourier},
    VOLUME = {65},
      YEAR = {2015},
    NUMBER = {3},
     PAGES = {1367--1379},
      ISSN = {0373-0956},
       DOI = {10.5802/aif.2959},
       URL = {https://doi.org/10.5802/aif.2959}
}

\end{document}